\documentclass[12pt,a4paper]{article}

\usepackage{amssymb}
\usepackage{amsmath, amsthm}
\usepackage{arydshln}
\usepackage{epsfig}
\usepackage{setspace}
\usepackage{comment} 
\usepackage[margin=1in]{geometry}

\def\RR{{\mathbb{R}}}
\def\ZZ{{\mathbb{Z}}}

\def\CC{{\mathbb{C}}}
\def\KK{{\mathbb{K}}}

\numberwithin{equation}{section}

\newcommand{\rank}{\mathop{\rm rank} }

\newcommand{\ncrank}{\mathop{\rm nc\mbox{-}rank} }

\newcommand{\diag}{\mathop{\rm diag} }

\newcommand{\trace}{\mathop{\rm tr} }

\newtheorem{Thm}{Theorem}[section]

\newtheorem{Lem}[Thm]{Lemma}
\newtheorem{Cor}[Thm]{Corollary}

\theoremstyle{definition}

\newtheorem{Rem}[Thm]{Remark}

\title{A scaling characterization of nc-rank \\
via unbounded gradient flow}
\author{Hiroshi Hirai \\
Graduate School of Mathematics,\\
Nagoya University, Nagoya, 464-8602, Japan.\\
\texttt{\normalsize hirai.hiroshi@math.nagoya-u.ac.jp}
}

\begin{document}

\maketitle
\begin{abstract}
Given a tuple of $n \times n$ complex matrices ${\cal A} = (A_1,A_2,\ldots, A_m)$, 
the linear symbolic matrix $A = A_1x_1 + A_2x_2 + \cdots + A_m x_m$
is nonsingular in the noncommutative sense
if and only if the completely positive operators 
$T_{\cal A} (X) = \sum_{i=1}^m A_i X A_i^{\dagger}$ and $T_{\cal A}^*(X) = \sum_{i=1}^m A_i^{\dagger} X A_i$
can be scaled to be doubly stochastic: For every $\epsilon > 0$ 
there are $g,h \in GL(n,\CC)$ such that $\|T_{g^{\dagger}{\cal A}h}(I)- I\| < \epsilon$,  
$\| T^*_{g^\dagger{\cal A}h}(I) - I\| < \epsilon$.
In this paper, we show a refinement: The noncommutative corank of 
$A$ is equal to one-half of the minimum residual 
$\|T_{g^{\dagger}{\cal A}h}(I) - I\|_1 + \|T^*_{g^{\dagger}{\cal A}h}(I) - I\|_1$ over all possible scalings $g^{\dagger}{\cal A}h$, 
where $\|\cdot \|_1$ is the trace norm. 
To show this, we interpret the residuals as gradients of a convex function on symmetric space $GL(n,\CC)/U_n$, and establish a general duality relation
of the minimum gradient-norm of a lower-unbounded convex function $f$ on 
 $GL(n,\CC)/U_n$ with an invariant Finsler metric, 
by utilizing the unbounded gradient flow of $f$ at infinity.
\end{abstract}
MSC classes: 90C25, 53C35

\section{Introduction}
Given an $m$ tuple of $n \times n$ matrices ${\cal A} = (A_1,A_2,\ldots,A_m)$ over field $\KK$, 
{\em Edmonds' problem}~\cite{Edmonds67}
asks to compute the rank of the matrix
\begin{equation}
A = A_1 x_1 + A_2 x_2 + \cdots + A_m x_m
\end{equation}
for indeterminates $x_1,x_2,\ldots,x_m$, 
where the rank is considered in the rational function field $\KK(x_1,x_2,\ldots,x_m)$.
A deterministic polynomial time algorithm for Edmonds' problem is not known, and 
is one of the prominent open problems in discrete mathematics 
and theoretical computer science; see e.g.,~\cite{Kabanets2004,Lovasz89}. 
Recently, a noncommutative version of Edmonds' problem was introduced 
by Ivanyos, Qiao, and Subrahmanyam~\cite{IQS2017} and is bringing about new developments. 
In this setting, 
the indeterminates are noncommutative $x_ix_j  \neq x_j x_i$, and the rank---the {\em noncommutative rank (nc-rank)} $\ncrank A$---is considered in the {\em free skew field}~\cite{Amitsur,Cohn} into which the nonncommutative polynomial ring $\KK\langle x_1,x_2,\ldots,x_m\rangle$ embeds.
Recent groundbreaking results~\cite{GGOW,HamadaHirai2021,IQS2018} 
show that the noncommutative Edmonds' problem 
can be solved in polynomial time.
The techniques used in this result have unexpected but fruitful interactions 
with many areas of pure and applied mathematics (noncommutative algebra, operator theory, representation theory, invariant theory, differential geometry, statistics, optimization, quantum information,...), 
and act as catalysts for interdisciplinary research; see e.g., \cite{BFGOWW_FOCS2019} and references therein.

One of the key ingredients of this development is the following characterization 
of the nc-rank due to Fortin and Reutenauer~\cite{FortinReutenauer04}.
For a vector subspace $U$ in $\KK^n$ 
(denoted by $U \leqslant \KK^n$), 
let ${\cal A} U$ denote the vector subspace 
that is the sum of $A_k U$ over $k=1,2,\ldots,m$. 
\begin{Thm}[\cite{FortinReutenauer04}]\label{thm:FR}
    $\displaystyle \ncrank A = n + \min_{U \leqslant \KK^n}  \dim {\cal A} U - \dim U$.
\end{Thm}
The first polynomial-time algorithm of nc-rank by Garg, Gurvits, Oliveira, and Wigderson~\cite{GGOW} 
is based on a connection to {\em operator scaling} (Gurvits~\cite{Gurvits2004})---a quantum generalization of the classical {\em matrix scaling} (Sinkhorn~\cite{Sinkhorn1964}). 
Suppose from now that $\KK$ is the field of complex numbers $\CC$.
The matrix tuple ${\cal A} = (A_1,A_2,\ldots,A_m)$ is 
associated with the completely positive operator 
$T_{\cal A}: \CC^{n\times n} \to \CC^{n \times n}$ and its dual $T^*_{\cal A}: \CC^{n\times n} \to \CC^{n \times n}$ by
\begin{equation}
T_{\cal A}(X) := \sum_{k=1}^m A_k X A_k^{\dagger}, \quad T_{\cal A}^*(X) := \sum_{k=1}^m A_k^{\dagger} X A_k \quad (X \in \CC^{n \times n}),
\end{equation}
where $(\cdot)^{\dagger}$ means the conjugate transpose.
For $g,h \in GL_n = GL(n,\CC)$, consider the scaled tuple
$g^{\dagger}{\cal A} h = (g^{\dagger}A_1h, g^{\dagger}A_2h,\ldots, g^{\dagger}A_mh)$, and the corresponding scaled operators 
$T_{g^{\dagger}{\cal A}h}$ and $T^*_{g^{\dagger}{\cal A}h}$.
The tuple ${\cal A}$ is said to be {\em (approximately doubly-stochastic) scalable} 
if for every $\epsilon > 0$ there are $g,h \in GL_n$ 
such that
\begin{equation}\label{eqn:epsilon}
\|T_{g^{\dagger}{\cal A}h}(I) - I\|^2_2 +   \|T^*_{g^{\dagger}{\cal A}h}(I) - I\|_2^2 < \epsilon^2,
\end{equation}
where $\|\cdot\|_2$ is the Frobenius norm. Then, Gurvits' characterization~\cite{Gurvits2004} of scalability is the following:
\begin{Thm}[\cite{Gurvits2004}]\label{thm:Gurvits}
    The following are equivalent:
 \begin{itemize}
     \item[(1)] ${\cal A}$ is scalable.
     \item[(2)] $\displaystyle \inf_{X \succ 0} \frac{\det \sum_{k=1}^m A_k X A_k^{\dagger}}{\det X} > 0$.
     \item[(3)] $\rank T_{\cal A}(X) \geq \rank X$ for every $X \succeq 0$.
 \end{itemize}   
\end{Thm}
By $X \succ 0$ (resp. $X \succeq 0$) we mean that $X$ is a positive definite (resp. semidefinite) Hermitian matrix.
In (3), 
one see that $\rank T_{\cal A}(X) = \dim {\cal A}U$ 
for the vector subspace $U$ spanned by column vectors of $g$ with $X = gg^\dagger$.
Hence, the condition (3) is equivalent to $\dim {\cal A} U \geq \dim U$ for all $U \leqslant \CC^n$. By Theorem~\ref{thm:FR}, we have:
\begin{Cor}[\cite{GGOW}]\label{cor:ncrank=n}
$\ncrank A = n$ if and only if ${\cal A}$ is scalable.
\end{Cor}
Thus, the optimization of the function in (2) links with  
the nc-nonsingularity determination of $A$, as did in~\cite{GGOW}.

The first main result of this paper is a more direct connection between nc-rank and scalability.
Let $\| \cdot \|_1$ denote the {\em trace norm}, that is, 
$\|X\|_1$ is the sum of singular values of $X$.
Then we prove that nc-corank is equal to one-half of 
the minimum residual of a scaling with respect to $\|\cdot\|_1$, 
generalizing Corollary~\ref{cor:ncrank=n}. 
\begin{Thm}\label{thm:main}
\begin{equation}\label{eqn:main}
n - \ncrank A = \frac{1}{2} \inf_{g,h \in GL_n}  \|T_{g^{\dagger}{\cal A}h} (I) - I\|_1 + \|T^*_{g^{\dagger}{\cal A}h} (I) - I\|_1.
\end{equation}
\end{Thm}

This relation (\ref{eqn:main}) is inspired by the duality theorem 
of the minimum gradient-norm of a geodesically convex function on a Hadamard manifold $M$,
due to Hirai and Sakabe~\cite{HiraiSakabe2024FOCS}. 
They established that the infimum of the gradient norm of a lower-unbounded convex function $f: M \to \RR$ 
is equal to the supremum of 
the negative of the {\em recession function} $f^{\infty}$ of $f$, 
and that these infimum and supremum are attained by the unbounded gradient-flow curve of $f$ at infinity.
In the operator scaling setting, 
the residual matrices $T_{g^{\dagger}{\cal A}h}(I) - I, T^*_{g^{\dagger}{\cal A}h}(I)-I$
of scaling $g^{\dagger}{\cal A}h$ can be interpreted 
as the (transported) gradient of a convex function $F$ on 
the product $P_n \times P_n$ of the manifold $P_n$ of positive definite matrices, 
where $F$ is a variant of the function in Theorem~\ref{thm:Gurvits}~(2).
From this, 
they deduced a formula~\cite[Theorem 4.19]{HiraiSakabe2024FOCS} of the scalability limit with respect to 
the Frobenius norm, i.e., the minimum possible $\epsilon > 0$ in (\ref{eqn:epsilon}).

Theorem~\ref{thm:main} 
is viewed as a trace-norm version, and is
obtained by extending their gradient-flow approach.
To capture estimates by the trace norm $\| \cdot \|_1$,
we consider a general $GL_n$-invariant norm on symmetric space 
$P_n \simeq GL_n/U_n$, 
regard it as a Finsler manifold, 
and consider the ``gradient flow" of the function in Theorem~\ref{thm:Gurvits}~(2). 
However, the norm in our setting may be non-differentiable, 
and the gradient flow cannot be defined directly.  
We then utilize a general framework~(Ambrosio, Gigli, and Savar\'e~\cite{AGS_GradientFlows}) of gradient flows in metric spaces, 
without differential structure,  
and establish a duality theorem (Theorem~\ref{thm:duality}) of the minimum gradient norm of a convex function on $P_n$ 
with respect to an invariant Finsler metric, 
which is placed as the second main result of this paper.
Then, by adapting this duality, we obtain Theorem~\ref{thm:main}.
Such an application of nondifferentiable gradient flows has never been seen before  
and is interesting in its own right.
Discretizing this gradient flow to develop a new algorithm 
and analysis for nc-rank and operator scaling 
is an important future research. 
Generalizing the presented duality to more general Finsler spaces
is also an interesting research; 
Berwald spaces may be candidates.

In Section~\ref{sec:preliminaries}, 
we introduce necessary backgrounds on gradient flows in metric spaces and 
symmetric space $P_n$ endowed with an invariant Finsler metric.
In Section~\ref{sec:duality}, we establish the duality theorem (Theorem~\ref{thm:duality}) 
of the minimum gradient-norm of a convex function. 
In Section~\ref{sec:proof}, we complete the proof of Theorem~\ref{thm:main}.  

\section{Preliminaries}\label{sec:preliminaries}
Let $\RR$ and $\RR_+$ denote the sets of real and nonnegative real numbers, respectively. 
Let ${\bf 1} = (1,1,\ldots,1)$ denote the all-one vector. 
For a positive integer $k$, let $[k] := \{1,2,\ldots,k\}$.
For an $n \times n$ matrix $A$, let $e^{A}$ denote the matrix exponential.
For a positive definite matrix $X$, let $\log X$ denote the unique matrix $Y$ with 
$X =e^{Y}$. For a vector $\lambda \in \RR^n$, let $\diag \lambda$ denote the $n \times n$ diagonal matrix whose diagonals are given by $\lambda_1,\lambda_2,\ldots,\lambda_n$ in order.

\subsection{Outline by Euclidean specialization}\label{subsec:outline}
As mentioned, our main result (Theorem~\ref{thm:main})
is deduced from a duality relation (Theorem~\ref{thm:duality}) of 
the minimum gradient-norm of a convex function on a certain Finsler manifold. 
Here we briefly explain a Euclidean specialization of this duality
and its proof method using gradient flows.
This will give motivation to several concepts and properties introduced 
in subsequent sections. 

Let $f:\RR^n \to \RR$ be a smooth convex function.
Our duality relation (Theorem~\ref{thm:duality}) is a manifold generalization of the following identity:
\begin{equation}\label{eqn:strong_duality_Euclidean}
\inf_{x\in \RR^n} \|\nabla f(x)\| = \sup_{u\in \RR^n:\|u\|\leq 1} - f^{\infty}(u),
\end{equation}
where $f^{\infty}:\RR^n \to \RR \cup \{\infty\}$ is the {\em recession function} of $f$ defined by
\begin{equation}\label{eqn:recession}
f^{\infty}(u) := \lim_{t \to \infty} \frac{f(x+tu)-f(x)}{t} \quad (u \in \RR^n).
\end{equation}
It is known that the $f^{\infty}$ is determined independently of $x \in \RR^n$; see \cite{Rockafellar} for the recession function.
Then the duality (\ref{eqn:strong_duality_Euclidean}) is seen by elementary convex analysis:
The LHS of (\ref{eqn:strong_duality_Euclidean}) is interpreted as the minimum-norm point problem 
on the closure of the gradient space 
$\overline{\nabla f(\RR^n)}$, which is written as 
$\overline{\nabla f(\RR^n)} = \{p \in \RR^n \mid \langle u,p\rangle \leq f^{\infty}(u)\ (u \in \RR^n:\|u\|=1)\}$.
Then (\ref{eqn:strong_duality_Euclidean}) is a consequence of e.g., Fenchel duality; 
see \cite[Section 3.3]{HiraiSakabe2024FOCS} for detail.

Although the relation (\ref{eqn:strong_duality_Euclidean})
naturally extends to manifold settings (as we will see), 
this proof method does not.    
Hirai and Sakabe~\cite{HiraiSakabe2024FOCS} developed 
the following constructive proof of (\ref{eqn:strong_duality_Euclidean}) that works on Hadamard manifolds.  
First, observe weak duality in (\ref{eqn:strong_duality_Euclidean}): 
\begin{description}
    \item[{\rm Weak duality:}] $\|\nabla f(x)\| \geq  - f^{\infty}(u)$ $(\forall x\in \RR^n, \forall u \in \RR^n:\|u\|\leq 1)$.
\end{description}
Indeed, $f^{\infty}(u) \geq \lim_{t \to 0} (f(x+tu)-f(x))/t = \langle \nabla f(x),u\rangle \geq - \|\nabla f(x)\|$, where the first inequality follows from convexity of $f$ and the second from Cauchy-Schwarz.

We construct $x,u$ attaining (\ref{eqn:strong_duality_Euclidean})
by the limit of the {\em gradient flow} of $f$:
Consider the initial value ODE problem
\begin{equation}
\dot x(t) = -\nabla f(x(t)), \quad x(0) = x_0,
\end{equation}
and its solution $x(t)$.
The construction utilizes the following well-known properties: 
\begin{description}
\item[{\rm Energy identity:}] $\displaystyle f(x(t)) - f(x_0) = - \int_{0}^t \|\nabla f(x(s))\|^2 ds$.
\item[{\rm Monotonicity:}] $t \to \|\nabla f(x(t))\|$ is nonincreasing.
\item[{\rm Triangle inequality:}] $\displaystyle \|x(t)-x_0\| \leq \int_{0}^t \|\dot x(s)\| ds = \int_{0}^t \|\nabla f(x(s))\| ds$.
\end{description}
Energy identity follows from integrating $\frac{d}{dt}f(x(t)) = \langle \nabla f(x(t)),\dot x(t)\rangle = - \|\nabla f(x(t))\|^2$. 
Monotonicity follows from $\frac{d}{dt} \|\nabla f(x(t))\|^2 
= - \langle \dot x(t),\nabla^2f(x(t))\dot x(t) \rangle \leq 0$, where the Hessian $\nabla^2f(x(t))$ is positive semidefinite due to convexity of $f$.
Triangle inequality follows from 
$\|x(t)-x_0\| = \| \int_{0}^t \dot x(s) ds\| \leq \int_{0}^t \|\dot x(s)\| ds$.
By the monotonicity, the length of gradient-flow curve $x(t)$ $(x \in [0,T))$ is finite for finite $T$.
This verifies that $x(t)$ is defined for $[0,\infty)$.

Now (\ref{eqn:strong_duality_Euclidean}) is proved as follows.
We may assume that $\kappa := \inf_{x \in \RR^n} \|\nabla f(x)\| > 0$.
Let $x_0 := 0$ (for simplicity). 
Consider the gradient-flow curve $x(t)$ $(t \in [0,\infty))$.
From energy identity it holds $f(x(t)) \to -\infty$ and then $\|x(t)\| \to \infty$.
Let $u(t) := x(t)/\|x(t)\|$. Then, for $r \leq \|x(t)\|$, we have 
\begin{eqnarray*}
\frac{f(ru(t))- f(0)}{r} &\leq& \frac{f(x(t))-f((0))}{\|x(t)\|} 
\leq - \frac{\int_{0}^t \|\nabla f(x(s))\|^2ds}{\int_{0}^t \|\nabla f(x(s))\|ds} \\
&\leq& - \frac{1}{t} \int_{0}^{t} \|\nabla f(x(s))\|ds  \leq  \lim_{t \to \infty} -\|\nabla f(x(t))\|,
\end{eqnarray*}
where the first inequality follows from convexity of $f$ 
along the line between $0$ and $x(t)$,
the second from energy identity and triangle inequality,
the third from Cauchy-Schwarz, and the fourth from monotonicity.  
Let $u^*$ be an accumulation point of $u(t)$.
Then $(f(ru^*)-f(0))/r \leq \lim_{t \to \infty} -\|\nabla f(x(t))\|$. By $r \to \infty$, 
we have $f^{\infty}(u^*) \leq \lim_{t \to \infty} -\|\nabla f(x(t))\|$.  With weak duality, we have
\[
\inf_{x\in \RR^n} \|\nabla f(x)\| \geq \sup_{u \in \RR^n:\|u\|\leq 1} - f^{\infty}(u) 
\geq - f^{\infty}(u^*) \geq \lim_{t \to \infty} \|\nabla f(x(t))\| \geq \inf_{x \in \RR^n} \|\nabla f(x)\|.
\]
All the inequalities must hold in equality. 
Moreover, $u(t)$ converges since the minimizer of $f^{\infty}$ over unit sphere is unique.
We then obtain a constructive strong duality by unbounded gradient flow~$x(t)$:
\[
\lim_{t \to \infty} \|\nabla f(x(t))\| = \inf_{x \in \RR^n} \|\nabla f(x)\| = \sup_{u \in \RR^n:\|u\|\leq 1}-f^{\infty}(u) 
= -f^{\infty} \left( \lim_{t \to \infty}\frac{x(t)}{\|x(t)\|} \right).
\]
The subsequent arguments are motivated by extending 
this approach.

\subsection{Gradient flow in metric space}
We follow \cite[Part I]{AGS_GradientFlows}.
Let $(M,d)$ be a complete metric space.
An {\em absolutely continuous curve} 
is a curve $\gamma:(a,b) \to M$ (with possibly $b =\infty$) such that 
there is a Lebesgue integrable function 
$m:(a,b) \to \RR$ such that
\begin{equation}\label{eqn:AC}
d(\gamma(s),\gamma(t)) \leq \int_{s}^t m(r) dr
\end{equation}
for any $s,t$ with $a < s \leq t < b$.
A {\em locally absolutely continuous curve} 
is a curve such that each point on the curve 
has an open interval on which the curve is absolutely continuous.
Note that the absolute continuity is stronger than uniform continuity.
For a (locally) absolutely continuous curve $\gamma:(a,b) \to M$, 
the {\em metric derivative} $|\gamma'|:(a,b) \to \RR_+$ is defined as 
\begin{equation}
    |\gamma'|(t)  := \lim_{s \to t} \frac{d(\gamma(s),\gamma(t)))}{|s -t| } \quad (t \in (a,b)).
\end{equation}
The limit exists for almost everywhere $t \in (a,b)$. The metric derivative $|\gamma'|$ can be chosen
as $m$ in the above (\ref{eqn:AC}).
Particularly, it holds
\begin{equation}
d(\gamma(s),\gamma(t)) \leq \int_{s}^{t} |\gamma'|(t) dt \quad (a < s \leq t < b). 
\end{equation}
This inequality can be interpreted as the triangular inequality if the right-hand side is defined as the length of the curve.

A $d$-geodesic is a curve $\gamma:[a,b] \to M$ such that
\begin{equation}
   d(\gamma(s),\gamma(t)) = \frac{|s-t|}{|a-b|} d(\gamma(a),\gamma(b)) \quad (a \leq s \leq t \leq b). 
\end{equation}
By a {\em $d$-geodesic ray} we mean a curve $\gamma:[a,\infty) \to M$ 
such that the restriction to every finite interval is a $d$-geodesic. 
A function $f:M \to \RR$ is called {\em (geodesically) $d$-convex} if for any $x,y \in M$ 
there exists a geodesic $\gamma:[0,1] \to M$ with $\gamma(0) = x$ and $\gamma(1) = y$ such that
\begin{equation}
 f(\gamma(t)) \leq (1-t) f(x) + tf(y) \quad (t \in [0,1]).
\end{equation}
The (global) {\em slope} $|\partial f|:M \to \RR_+$ of a convex function $f$ is defined by
%
\begin{eqnarray}
|\partial f|(x) &:= & \sup_{y \in M \setminus \{x\}} \frac{\max \{0, f(x) - f(y)\}}{d(x,y)} \nonumber \\
&=& \sup_{\gamma} \lim_{t \to 0}  \frac{\max \{0, f(\gamma(0)) - f(\gamma(t))\}}{t} \quad (x \in M), \label{eqn:slope} 
\end{eqnarray}
where the supremum in (\ref{eqn:slope}) is taken over all geodesics $\gamma:[0,a] \to M$ 
such that $\gamma(0) = x$ and $a = d(\gamma(0),\gamma(a)) > 0$, 
and the equality follows from convexity of $f$.

A {\em curve of maximal slope (a.k.a. gradient-flow curve)} of $f$ with initial point $x_0 \in M$ is 
a locally absolutely continuous curve $x: (0,\infty) \to M$ 
satisfying $\lim_{t \to +0}x(t) = x_0$ and
\begin{equation}
\frac{d}{dt} f(x(t)) \leq - \frac{1}{2} |x'|^2(t) - \frac{1}{2} |\partial f|^2(x(t)) \quad ({\rm a.e.}\ t \in (0,\infty)).
\end{equation}
It is known \cite[Remark 1.3.3]{AGS_GradientFlows} that the inequality becomes equality a.e. and it holds
\begin{equation}\label{eqn:equal}
|x'|^2(t) = |\partial f|^2(x(t)) = - \frac{d}{dt} f(x(t)) \quad ({\rm a.e.}\ t \in (0,\infty)),
\end{equation}
and the the following energy identity holds
\begin{equation}\label{eqn:energy}
f(x(b)) - f(x(a)) =  -\int_{a}^b  \frac{1}{2} |x'|^2(t) + \frac{1}{2} |\partial f|^2(x(t)) dt 
\quad (0 \leq a \leq b < \infty).
\end{equation}
A curve of maximal slope of $f$ is constructed by the limit of the implicit Euler scheme,
called a {\em generalized minimizing movement}~\cite[Chapter 2]{AGS_GradientFlows}. 
\begin{Thm}[{\cite[Theorem 2.4.15]{AGS_GradientFlows}}]\label{thm:GMM}
    Suppose that $f: M \to \RR$ is lower-semicontinuous $d$-convex and 
    every bounded subset of a sublevel set of $f$ is relatively compact. Then, for every $x_0 \in M$, there exists a curve of maximal slope $x:(0,\infty) \to M$ of $f$ with $\lim_{t \to +0}x(t)= x_0$ such that 
    \begin{itemize}
    \item[{\rm (1)}] $t \mapsto f(x(t))$ is convex, and 
    \item[{\rm (2)}] $t \mapsto |\partial f|(x(t))$ is nonincreasing and right continuous.
    \end{itemize}
\end{Thm}

\subsection{Symmetric space with invariant Finsler metric}

Let $v$ be a norm on $\RR^n$, 
that is, it satisfies $v(x) \geq 0$, $v(x) = 0 \Leftrightarrow x = 0$, $v(\alpha x) = |\alpha| v(x)$, and $v(x+y) \leq v(x) + v(y)$ for $x,y \in \RR^n, \alpha \in \RR$. 
The dual space $(\RR^n)^*$ is identified with $\RR^n$ by 
$x \mapsto x^{\top}y$ $(y \in \RR^n)$.
The dual norm $v^*$ of $v$ is defined by 
$v^*(x) := \max \{ x^{\top}y \mid y \in \RR^n:v(y) \leq 1\}$ for  $x \in \RR^n$.
Suppose further that $v$ is invariant under coordinate permutations, that is, 
it holds $v(x) = v(\sigma x)$
for every $x \in \RR^n$ and every permutation matrix $\sigma$.
Then, the dual norm $v^*$ is also invariant under coordinate permutations.

Let $S_n$ denote the real vector space of $n \times n$ Hermitian matrices.
The dual space of $S_n$ is identified with $S_n$ by $X \mapsto \trace X Y$ $(Y \in S_n)$.
Let $U_n$ denote the group of $n \times n$ unitary matrices. 
For a permutation invariant norm $v$ on $\RR^n$, 
we obtain a unitarily invariant norm $\| \cdot \|_v$ on $S_n$:
\begin{equation}
\|H \|_v := v(\lambda) \quad (H \in S_n),
\end{equation}
where $H = u \diag \lambda u^\dagger$ for $\lambda \in \RR^n$ and $u \in U_n$.
The dual norm $\|\cdot \|^*_v$ of $\| \cdot \|_v$ 
is defined by $\|H \|^*_{v} := \max \{ \trace H X \mid X \in S_n:\|X\|_v \leq 1\}$.
The diagonal projection $D: S_n \to \RR^n$ is defined by
\begin{equation}\label{eqn:diagonal}
D(H)_{i} := H_{ii} \quad (H \in S_n, i \in [n] ). 
\end{equation}
The following are well-known, and can be found in \cite{Davis1957,Lewis1996} (for unitarily invariant convex functions).
\begin{Lem}\label{lem:norm}
\begin{itemize}
\item[{\rm (1)}] $\|\cdot \|_v$ is a norm on $S_n$.
\item[{\rm (2)}] $\|\cdot \|_v^* = \|\cdot \|_{v^*}$.
\item[{\rm (3)}] $\| H \|_v \geq v(D(H))$ for $H \in S_n$.
\end{itemize}
\end{Lem}
We give a proof in the Appendix for convenience.
If $v$ is the $l_p$-norm $\| \cdot \|_p$ for $p \in [1,\infty]$, 
then $\|\cdot \|_v$, denoted by $\| \cdot \|_p$, is known as the {\em Schatten $p$-norm}.
Our particular interest is the trace norm $\|\cdot \|_1$, which is the dual norm of spectral norm $\| \cdot \|_{\infty}$ (by above (2)).
\begin{equation}
\|\cdot \|_1 = \|\cdot \|_{\infty}^*. 
\end{equation}

We next consider the manifold $P_n \subseteq S_n$ of positive definite matrices. 
We utilize elementary notions of Riemannian geometry; see e.g., \cite{Boumal_Book,Sakai1996}. 
The tangent space $T_X$ at each point $X \in P_n$ is identified with $S_n$, and the cotangent space $T_X^*$ 
is also identified with $S_n$ by $F \mapsto \trace FH$ $(H \in T_X = S_n)$.
A differentiable function $f: P_n \to \RR$ gives rise 
to a point $df(X) \in  T_X^*$ by  $df(X)(H) := \frac{d}{dt} \mid_{t = 0} f(\gamma(t))$, where $\gamma:[0,a) \to P_n$ is
an arbitrary smooth curve with $\gamma(0) = X$ and $\dot \gamma(0) = H \in T_X$.
It is called the {\em differential} of $f$ at $X$. 
The standard Riemannian metric $\langle ,\rangle$ on $P_n$
is defined by $\langle H,H' \rangle_X := \trace X^{-1} H X^{-1} H'$ for $X \in P_n$ and $H,H' \in T_X$.
The corresponding norm is particularly denoted by $\|\cdot \|_{2,X} := \sqrt{\langle \cdot, \cdot\rangle_{X}}$. 
$GL_n$ acts isometrically and transitively on $P_n$ by $(g,X) \mapsto gXg^\dagger$. 
The stabilizer at $I$ is the unitary group $U_n$. 
So $P_n$ is viewed as the symmetric space $GL_n/U_n$.
This is a representative example of {\em Hadamard manifolds}, 
that is, simply-connected Riemannian manifolds having 
nonpositive sectional curvature everywhere; see \cite[Part II, Chapter 10]{BridsonHaefliger1999}.
The Riemannian length of a curve $\gamma:[a,b] \to P_n$ is defined as 
$\int_{a}^b \|\dot \gamma (t)\|_{2,\gamma(t)} dt$, where 
$\dot \gamma(t) \in T_{\gamma(t)}$ denotes the tangent vector of $\gamma$. 
The distance $d_2(X,Y)$ between $X$ and $Y$ 
is the infimum of the length of a curve $\gamma:[a,b] \to M$ with $\gamma (a) = X$ and $\gamma(b)=Y$. 
It is known that the infimum is (uniquely) attained by a curve with form $t \mapsto g e^{t H} g^{\dagger}$ that is a $d_2$-geodesic.
Specifically, for $H \in T_X$ the unique $d_2$-geodesic ray $t \mapsto \gamma_{X,H}(t)$ with $\gamma_{X,H}(0) = X$ and $\dot \gamma_{X,H}(0) = H$
is given by
\begin{equation}
\gamma_{X,H}(t) = X^{1/2} e^{-t X^{-1/2}HX^{-1/2}} X^{1/2} \quad (t \in [0,\infty)).
\end{equation}
$\|H\|_{2,X}$ is called the {\em speed} of $\gamma_{X,H}$. 
If $\|H\|_{2,X} = 1$, then $\gamma_{X,H}$ is called a {\em unit-speed} $d_2$-geodesic ray.
The distance $d_2(X,Y)$ between two points $X,Y\in P_n$ is explicitly calculated by
\begin{equation}
d_2(X,Y) = d_2(I, X^{-1/2}YX^{-1/2}) = \| \log X^{-1/2} Y X^{-1/2} \|_2. 
\end{equation}

Fix a unitarily invariant norm $\|\cdot\|_v$ on $S_n = T_I$. 
We introduce a $GL_n$-invariant Finsler metric on $P_n$ as follows. 
For each $X \in P_n$,
define norm $\|\cdot\|_{v,X}$ on $T_X$ by
\begin{equation}
\|H\|_{v,X} :=  \|X^{-1/2} H X^{-1/2} \|_v \quad (H \in T_X).
\end{equation}
The dual norm  $\|\cdot\|^*_{v,X}$ on $T_X^*$ is given by
\begin{equation}
\|F\|^*_{v,X} := \|X^{1/2} F X^{1/2} \|_v^* \quad (F \in T_X^*).
\end{equation}
Observe that the norm $\| \cdot \|_v$ is invariant under the action $(g,X) \mapsto gXg^{\dagger}$.
As in the Riemannian case, the length $L[\gamma]$ of a curve $\gamma:[a,b] \to P_n$ is defined by
\begin{equation}
    L[\gamma] := \int_{a}^b \|\dot \gamma(t)\|_{v, \gamma(t)} dt.
\end{equation}
The distance $d_v(X,Y)$ between two points $X,Y$ 
is defined as the infimum of a curve $\gamma:[a,b] \to M$
connecting $X = \gamma(a)$ and $Y = \gamma(b)$.
Now we obtain a metric space $(P_n,d_v)$. 
Bhatia~\cite{Bhatia2003} and
Friedland and Freitas~\cite{FriedlandFreitas2004} studied this metric space and characterized the metric structure, as follows.
\begin{Thm}[{\cite[(23)]{Bhatia2003} and \cite[Theorem 3.2]{FriedlandFreitas2004}; see also \cite[Section 6.4]{Bhatia_PositiveDefiniteMatrices}}]\label{thm:Finsler}
$(P_n,d_v)$ is a complete geodesic metric space, where
$d_v$ is given by
\begin{equation}
d_v(X,Y) = d_v(I, X^{-1/2}YX^{-1/2}) = \| \log (X^{-1/2} Y X^{-1/2}) \|_{v}.
\end{equation}
\end{Thm}
In \cite{FriedlandFreitas2004}, this theorem is stated for the Schatten $p$-norm,
although its proof works for any unitarily invariant norm. 
%

\begin{Cor}\label{cor:convexity}
Any $d_2$-geodesic in $P_n$ is also $d_v$-geodesic.
Any $d_2$-convex function $f:P_n\to \RR$ is also $d_v$-convex.
\end{Cor}
The metric derivatives and slopes in $(P_n,d_v)$ 
are given as follows.
\begin{Lem}\label{lem:correspondence}
\begin{itemize}
\item[(1)] For an absolutely continuous curve $\gamma:(a,b) \to P_n$, it holds
\[
|\gamma'|(t) = \| \dot \gamma (t)\|_{v,\gamma(t)} \quad ({\rm a.e.}\ t \in (a,b)).
\] 
\item[(2)] For a differentiable $d_v$-convex function $f:P_n \to \RR$, it holds 
\[
|\partial f|(X) = \| df(X)\|^*_{v,X} \quad (X \in P_n).
\]
\end{itemize}
\end{Lem}
\begin{proof}
(1). We can assume that 
$\gamma(0) = I$ and 
$\gamma$ is differentiable at $0 \in (a,b)$.
Also we can assume that $\gamma(t) = e^{H(t)}$
for absolutely continuous curve $t \mapsto H(t) \in S_n$ such that $H(0) = 0$  and it is differentiable at $t=0$.
From $\gamma(t) = I + H(t) + \frac{1}{2!}H(t)^2 +\cdots$ and $H(0) = 0$, we have $\dot \gamma(0) = \dot H(0)$.
Then, by the definition of the metric derivative, it holds 
\[
|\gamma'|(0) = \lim_{s \to 0} \frac{d_v(\gamma(s),I)}{|s|}=\lim_{s \to 0} \frac{\| H(s) \|_v}{|s|} 
= \lim_{s \to 0} \left\| \frac{H(s) - H(0)}{s} \right\|_v = \| \dot H(0) \|_v = \|\dot \gamma(0)\|_{v,I}.
\]

(2). By the definition (\ref{eqn:slope}) of the slope and the bijectivity of the exponential map, 
it holds
\begin{eqnarray*}
|\partial f|(X) 
&=& 
\sup_{H \in S_n: \|H\|_{v,H} =1} \lim_{t \to 0} \frac{\max \{0, f(X)- f(\gamma_{X,H}(t))\}}{t}\\
&=& \sup_{H \in S_n: \|H\|_{v,H}=1} \max\{0, df(X)(H)\} = \|df(X)\|_{v,X}^*, 
\end{eqnarray*}
where we use $d_v$-convexity of $f$ for the second equality and $df(X)(-H) = - df(X)(H)$ for the last equality.
\end{proof}

\section{Duality of minimum gradient-norm}~\label{sec:duality}
For a $d_2$-convex function $f:P_n \to \RR$, 
the {\em recession function (asymptotic slope)}~\cite{Hirai_Hadamard2022,KLM2009JDG} $f^{\infty}_X:T_X \to \RR \cup \{\infty\}$ is defined by
\begin{equation}
f^{\infty}_X(H) := \lim_{t \to \infty} \frac{f(\gamma_{X,H}(t))- f(X)}{t} \quad (H \in T_X).
\end{equation}
By convexity of $f$, the function in the limit is monotone nondecreasing, 
and the limit exists in $\RR \cup \{\infty\}$.
Also it is positively homogeneous: $f^{\infty}_X(aH) = a f^{\infty}_X(H)$ for $a\geq 0$.
Therefore, $f^{\infty}_X$ is determined by the values on 
the unit sphere $S_X := \{ H \in T_X \mid \|H\|_{2,X} = 1\}$.
It is known~\cite[Lemma 2.10]{KleinerLeeb2006} that for $H \in S_X$, $H' \in S_{X'}$ 
it holds $f^{\infty}_X(H) = f^{\infty}_{X'}(H')$ 
if the corresponding unit-speed $d_2$-geodesic rays $t \mapsto \gamma_{X,H}(t)$  
and $t \mapsto \gamma_{X',H'}(t)$ are {\em asymptotic}
that is,
$\sup_{t \geq 0} d_2(\gamma_{X,H}(t), \gamma_{X',H'}(t)) < \infty$.
See also \cite[Section 2.3]{HiraiSakabe2024FOCS}.
The asymptotic relation is an equivalence relation $\sim$ on the set of all unit-speed $d_2$-geodesic rays.
It is naturally extended on arbitrary $d_2$-geodesic rays  
by: $\gamma \sim \gamma'$ if the speeds of $\gamma$ and $\gamma'$ are the same and 
their corresponding unit-speed geodesics are asymptotic.
The resulting space of equivalence classes of $d_2$-geodesic rays 
are known as (the Euclidean cone of) the {\em boundary of $M$ at infinity}.
This space is identified with $T_X$ for arbitrary $X\in P_n$ 
since all geodesics issuing at $X$ are the representatives of this space.
See \cite[Part II, Chapters 8--10]{BridsonHaefliger1999} for the boundary (of arbitrary CAT(0) spaces).

We choose the tangent space $T_I = S_n$ at the identity matrix $I \in P_n$ 
as a specified set of the representatives.
We may denote $f^{\infty}_I$ simply by $f^{\infty}$ 
and regard it as $S_{n} \to \RR \cup \{\infty\}$.
For a geodesic ray $t \mapsto \gamma_{X,H}(t)$ 
issuing at $X$, the unique geodesic issuing at $I$ asymptotic to $\gamma_{X,H}$ 
is given by $t \mapsto \gamma_{I, uX^{-1/2}HX^{-1/2}u^\dagger}(t)$ for some unitary matrix $u$.\footnote{This can be seen by adapting the following:
For $\lambda \in \RR^n$ with $\lambda_1 \geq \lambda_2 \geq \cdots \geq \lambda_n$, it is not difficult to see 
that $t \to e^{t \diag \lambda}$ and $t \to b e^{t \diag \lambda} b^{\dagger}$ are asymptotic if $b \in GL_n$ is upper-triangular. 
For $g \in GL_n$, consider QR-factorization 
$g = u b$ for unitary $u$ and upper-triangular $b$. 
Then the geodesic ray $t \mapsto ge^{t\diag \lambda}g^{\dagger} = u b e^{t \diag \lambda}b^{\dagger}u^{\dagger}$ is asymptotic to $t \to ue^{t \diag \lambda}u^{\dagger} = e^{t u \diag \lambda u^{\dagger}}$.
}
By the unitary invariance of the norm, it holds $\|H\|_{v,X} = \|X^{-1/2}HX^{-1/2}\|_v = \|uX^{-1/2}HX^{-1/2}u^\dagger\|_v$.
Hence we have
\begin{equation}\label{eqn:f_X^infty=f^infty}
\inf_{H \in T_X: \|H\|_{v,X}\leq 1} f_X^{\infty}(H) = \inf_{H \in S_n:\|H\|_v \leq 1} f^{\infty}(H). 
\end{equation}
Now, we formulate the main result of this section.
\begin{Thm}\label{thm:duality} Let $f: P_n \to \RR$ be a differentiable $d_2$-convex function. 
Then it holds
\begin{equation}
\inf_{X \in P_n} \|df(X)\|^*_{v,X}  = \sup_{H \in S_n:\|H\|_v \leq 1}  - f^{\infty}(H).
\end{equation}
Suppose that  $\kappa^* := \inf_{X \in P_n} \|df(X)\|^*_{v,X} > 0$.
There is a curve of maximal slope $X:(0,\infty) \to P_n$ of $f$ such that
\begin{itemize}
\item[(1)] $\displaystyle \lim_{t \to \infty} \|df(X(t))\|^*_{v,X(t)} = \kappa^*$, and  
\item[(2)]  $f^{\infty}(H^*) = - \kappa^*$ for any accumulating point $H^*$ of 
$H(t) := (\log X(t))/d_v(I,X(t))$.
\end{itemize} 
\end{Thm}

The proof goes analogously as in \cite[Theorem 3.1]{HiraiSakabe2024FOCS} and in Section~\ref{subsec:outline}.
\begin{Lem}[Weak duality]\label{lem:weakduality}
\begin{equation}
\inf_{X \in P_n} \|df(X)\|^*_{v,X}  \geq \sup_{H \in S_n:\|H\|_v \leq 1}  - f^{\infty}(H).
\end{equation}
\end{Lem}
\begin{proof}
From convexity of $f$, it holds
\begin{eqnarray*}
f_X^{\infty}(H) &= & \lim_{t \to \infty} \frac{f(\gamma_{X,H}(t)) - f(X)}{t} \geq  \lim_{t \to 0} \frac{f(\gamma_{X,H}(t)) - f(X)}{t} = df(X)(H) \\ 
&\geq & - \|df(X)\|_{v,X}^*\|H\|_{v,X} \geq  - \|df(X)\|_{v,X}^*,
\end{eqnarray*}
where we used $\|H\|_v \leq 1$ for the last inequality.
Therefore, by (\ref{eqn:f_X^infty=f^infty}), we have 
\[
\| df(X)\|_{v,X}^* \geq \sup_{H \in T_X:\|H\|_{v,X}\leq 1 } -f_X^{\infty}(H) = \sup_{H \in S_n:\|H\|_v\leq 1 } -f^{\infty}(H)
\]
for any $X \in P_n$.
\end{proof}

\begin{proof}[Proof of Theorem~\ref{thm:duality}]
We may consider the case of $\kappa^* > 0$.
By Theorem~\ref{thm:Finsler} and Corollary~\ref{cor:convexity}, 
Theorem~\ref{thm:GMM} is applicable to the $d_v$-convex function $f$ 
on complete metric space $(P_n,d_v)$ in which every bounded closed set is compact.
Consider a curve $X:(0,\infty) \to P_n$ of a maximal slope of $f$ 
with $\lim_{t \to +0}x(t) = I$ in Theorem~\ref{thm:GMM}.
Recall Lemma~\ref{lem:correspondence} for the correspondence $ |\partial f|(X) = \|df(X)\|^*_{v,X}$ and $|X'|(t) = \|\dot X(t)\|_{v,X(t)}$.
Let $\kappa := \lim_{t \to \infty} |\partial f|(X(t))$; the limit exists by Theorem~\ref{thm:GMM}~(2).
It is obvious that $\kappa \geq \kappa^* (> 0)$. By (\ref{eqn:equal}), (\ref{eqn:energy}), and the triangle inequality, we have
\begin{eqnarray}
f(X(t)) - f(I) &=& - \int_{0}^t |\partial f|^2(X(s)) ds \leq - \kappa^2 t,  \label{eqn:1} \\
d(X(t),I) & \leq & \int_{0}^t |X'|(s) ds = \int_{0}^t |\partial f|(X(s)) ds. \label{eqn:2}
\end{eqnarray}
It necessarily holds $d(X(t),I) \to \infty$. 
Otherwise, by (\ref{eqn:1}), 
$(X(t))_{t \in \RR_+}$ has an accumulation point $X^*$ such that $f(X^*) = - \infty$, 
contradicting the continuity of $f$.

Choose $H(t) \in T_I = S_n$ such that $\|H(t)\|_v = 1$ and $X(t) = e^{d_v(X(t),I)H(t)}$. 
For all $r \in (0,d_v(X(t),I)]$ it holds
\begin{eqnarray*}
\frac{f(e^{r H(t)}) - f(I)}{r} & \leq & 
\frac{f(X(t)) - f(I)}{d_v (X(t),I)} \leq \frac{- \int_{0}^t |\partial f|^2(X(s))  ds}{ \int_{0}^t |\partial f|(X(s)) ds} \\ 
&\leq& - \frac{1}{t}\int_{0}^t |\partial f|(X(s)) ds  \leq  - \kappa,
\end{eqnarray*}
where the first inequality follows from $d_v$-convexity of $f$ along $d_v$-geodesic $\tau \mapsto e^{\tau H(t)}$, 
the second from (\ref{eqn:1}) and (\ref{eqn:2}), the third 
from Cauchy-Schwarz inequality $(\int f \cdot 1 ds)^2 \leq \int f^2 ds \int 1^2 ds$, and the last from Theorem~\ref{thm:GMM} (2).

Choose any subsequence of $(t_i)_{i=1,2,\ldots}$ such that $H(t_i) \to H^*$ and $d(X(t_i),I) \to \infty$. Then we have
\[
\frac{f(e^{r H^*}) - f(I)}{r} \leq - \kappa \quad (r \in (0,\infty)). 
\]
By $r \to \infty$, we have $f^{\infty}(H^*) \leq - \kappa$. Hence, by weak duality (Theorem~\ref{lem:weakduality}) we have 
\[ \inf_{H \in S_n: \|H\|_{v} \leq 1}f^{\infty}(H) \leq 
f^{\infty}(H^*) \leq - \kappa \leq - \kappa^* = \sup_{X \in P_n}- |\partial f|(X)\ \leq \inf_{H \in S_n:\|H\|_v \leq 1}f^{\infty}(H).  
\]
\end{proof}

\section{Proof of Theorem~\ref{thm:main}}\label{sec:proof}

We first show the weak duality in Theorem~\ref{thm:main}.
The following is a generalization of \cite[Lemma 2.2]{HayashiHiraiSakabe} in matrix scaling.
\begin{Lem}\label{lem:weakduality2}
For $g,h \in GL_n$ and $U \leqslant \CC^n$, it holds
\begin{equation}\label{eqn:weakduality}
\|T_{g^{\dagger}{\cal A}h}(I) - I \|_1 + \|T^*_{g^{\dagger}{\cal A}h}(I) - I \|_1 \geq 2(\dim U - \dim {\cal A}U).
\end{equation}
\end{Lem}
\begin{proof}
Since $\dim h^{-1} U = \dim U$ and $\dim g^{\dagger}{\cal A}U = \dim {\cal A}U$, 
by replacing $g^{\dagger}{\cal A}h$ with ${\cal A}$, 
we may assume that $g= h = I$. 
Let $k := \dim U$ and $\ell := \dim {\cal A}U$.
Let $u \in U_n$ be a unitary matrix such that 
the first $k$ columns of $u$ span $U$.
Let $v \in U_n$ be a unitary matrix such that 
the first $n-\ell$ columns of $v$ span the orthogonal complement of ${\cal A}U$
with respect to product $(x,y) \mapsto x^{\dagger}y$.

Let ${\cal B} := v^{\dagger} {\cal A} u$. 
Then, each matrix $B_k$ in ${\cal B}$ has an $(n-\ell) \times k$ zero submatrix in the upper-left corner.  
Also $\|T_{{\cal A}}(I) - I \|_1 =\|T_{{\cal B}}(I) - I \|_1$ and 
$\|T^*_{{\cal A}}(I) - I \|_1 =\|T^*_{{\cal B}}(I) - I \|_1$.
Define $n \times n$ nonnegative matrix $C = (C_{ij})$ by
\[
C_{ij} := \sum_{k=1}^m (B_k)_{ij}(B^{\dagger}_k)_{ij} = \sum_{k=1}^m |(B_k)_{ij}|^{2}.
\]
By using the diagonal projection $D$ in (\ref{eqn:diagonal}), we observe
\[
D(T_{\cal B}(I) - I) = C{\bf 1} - {\bf 1},\quad D(T^*_{\cal B}(I) - I) = C^{\top}{\bf 1} - {\bf 1}.  
\]
Therefore, by Lemma~\ref{lem:norm}~(3), we have 
\begin{equation}\label{eqn:above}
\|T_{\cal B}(I) - I \|_1+ \| T^*_{\cal B}(I) - I \|_1  \geq \| C{\bf 1} - {\bf 1}\|_1 + \|C^{\top}{\bf 1} - {\bf 1} \|_1. 
\end{equation}
We show that RHS in (\ref{eqn:above}) $\geq$ RHS in (\ref{eqn:weakduality}). 
Now $C$ has an $(n- \ell) \times k$ zero submatrix of $C$ in the left-upper corner.  
The rest is the same as in  \cite[Lemma 2.2]{HayashiHiraiSakabe}:
Let $p := C{\bf 1}$ and $q := C^{\top}{\bf 1}$.
Since $C_{ij} = 0$ for $i \in [n -\ell],j \in [k]$, 
it holds
\begin{equation}\label{eqn:second}
\sum_{i=1}^{n-\ell} p_i \leq \sum_{j=k+1}^{n} q_j,\quad \sum_{j=1}^{k} q_j \leq \sum_{i=n - \ell+ 1}^{n} p_i.
\end{equation}
Therefore we have
\begin{eqnarray*}
 &&\| C{\bf 1} - {\bf 1}\|_1 + \|C^{\top}{\bf 1} - {\bf 1} \|_1 = \sum_{i=1}^n |p_i - 1 | + \sum_{j=1}^n |q_j - 1 | \\
 && =  \sum_{i=1}^{n-\ell} |p_i - 1 | +\sum_{i=n-\ell+1}^n |p_i - 1 | + \sum_{j=1}^{k} |q_j - 1 | + \sum_{j=k+ 1}^{n} |q_j - 1 | \\
 && \geq (n-\ell) - \sum_{i=1}^{n-\ell} p_i + \sum_{i=n-\ell+1}^n p_i - \ell + k - \sum_{j=1}^{k} q_j +\sum_{j=k+ 1}^{n} q_j - (n-k) \\
 && \geq 2(k- \ell) = 2 (\dim U - \dim {\cal A}U),
\end{eqnarray*}
where the first inequality follows from the triangle inequality 
and the second from (\ref{eqn:second}).
\end{proof}

Suppose ${\cal A}\CC^n = \CC^n$, which is equivalent to $\rank (A_1\ A_2\ \cdots \ A_m) = n$. 
Then it holds $\det \sum_{k=1}^m A_k X A_k^{\dagger}  = \det (A_1X\ A_2X\ \cdots \ A_mX)(A_1\ A_2\ \cdots \ A_m)^{\dagger} > 0$ for each $X \in P_n$. 
Define $f: P_n \to \RR$ by
\begin{equation}
f(X) := \log \det \sum_{k=1}^m A_k X A_k^{\dagger} - \log \det X \quad (X \in P_n).
\end{equation} 

\begin{Lem}[{See e.g., \cite{BFGOWW_FOCS2019}}]
$f$ is $d_2$-convex.
\end{Lem}
\begin{proof}[Sketch of proof]
By Binet-Cauchy formula, $f(ge^{\diag \lambda}g^{\dagger})$ is written as
\[
f(ge^{\diag \lambda}g^{\dagger}) = \log \sum_{j=1}^N a_j e^{c_j^{\top} \lambda} - {\bf 1}^{\top} \lambda - \det gg^{\dagger}
\]
for some $a_j > 0$, $c_j \in \ZZ^{n}$ $(j=1,2,\ldots,N)$. 
Then $\RR^n \ni \lambda \mapsto f(ge^{\diag \lambda}g^{\dagger})$ is convex, from which $d_2$-convexity of $f$ follows. 
\end{proof}
The differential $df$ of $f$ provides the residual of 
a scaling as follows.  
\begin{Lem}\label{lem:df}
Suppose that ${\cal A}\CC^n = \CC^n$.
For $X = hh^{\dagger} \in P_n$, it holds
\[
h^{\dagger} df(X) h = T^*_{g^{\dagger}{\cal A}h}(I) -I \quad (\in T_I^* = S_n),
\]
where $g \in GL_n$ satisfies $T_{g^{\dagger}{\cal A}h}(I) = I$.
\end{Lem}
\begin{proof}
From $T_{g^{\dagger}{\cal A}h}(I) = I$, it holds $(gg^{\dagger})^{-1} =  \sum_k A_kXA_k^{\dagger} =:B$. 
\begin{eqnarray*}
df(X)(H) &= &\frac{d}{dt} \mid_{t = 0} \log \det \sum_k A_k (X+tH) A_k^{\dagger} - \log \det (X+tH) \\
&=& \trace B^{-1}\sum_k A_k HA_k^{\dagger} - \trace X^{-1}H = \trace \sum_k A_k^{\dagger}B^{-1} A_k H - \trace X^{-1}H\\
&=&  \trace \left( \sum_{k}  A_k^{\dagger} gg^{\dagger} A_k  - X^{-1} \right)H.
\end{eqnarray*}
Hence we have $df(X) = \sum_{k}  A_k^{\dagger} gg^{\dagger} A_k - X^{-1}$, 
and $h^{\dagger}df(X)h = \sum_{k}  h^{\dagger}A_k^{\dagger} gg^{\dagger} A_kh - I$.
\end{proof}

We next compute the recession function $f^{\infty}$ of $f$. 
The following is a variant in \cite[Proposition 3.5 (1)]{Hirai_Hadamard2022}.
\begin{Lem}\label{lem:f^infty}
Suppose that ${\cal A}\CC^n = \CC^n$.
For $H \in S_n$, suppose that $H = u \diag \lambda u^{\dagger}$ for $u \in U_n$ and 
$\lambda \in \RR^n$ with $\lambda_1 \geq \lambda_2 \geq \cdots \geq \lambda_n$. Let $U_i$ denote the vector subspace spanned by the first $i$ columns of $u$. Then it holds
\begin{equation}
f^{\infty}(H) = \sum_{i=1}^n (\lambda_i - \lambda_{i+1}) (  \dim {\cal A}U_i - \dim U_i), 
\end{equation}
where $\lambda_{n+1} := 0$. 
\end{Lem}
Namely, $f$ is {\em asymptotically submodular} in the sense of \cite{Hirai_Hadamard2022} that  
$f^{\infty}$ coincides with the {\em Lov\'asz extension} of 
the submodular function $U \mapsto \dim {\cal A}U - \dim U$ on the modular lattice of vector subspaces.
This fact was announced in \cite[Remark 3.8 (1)]{Hirai_Hadamard2022}.
\begin{proof}
By definition of $f^{\infty}$, we have
\begin{equation}
    f^{\infty}(H) = \lim_{t \to \infty} \frac{f(u e^{t\diag \lambda} u^{\dagger})}{t} = - {\bf 1}^{\top} \lambda + \lim_{t \to \infty}\frac{1}{t} 
    \log \det \sum_{k=1}^m (A_k u)e^{t \diag \lambda} (A_ku)^{\dagger}. 
\end{equation}
Here  
\[
 \sum_{k=1}^m (A_k u)e^{t \diag \lambda} (A_ku)^{\dagger} = 
 (A_1u\  \cdots A_mu)   
 \left(
 \begin{array}{ccc}
 e^{t\diag \lambda} &  &     \\
     & \ddots    &      \\
     &      &    e^{t\diag \lambda}   
 \end{array}
 \right)
 \left(
 \begin{array}{c}
 (A_1u)^{\dagger} \\
 \vdots \\
 (A_m u)^{\dagger} 
 \end{array}
 \right).
\]
We rearrange the matrices in the RHS as follows.
Let
$M[i_1,i_2,\ldots,i_k]$ denote 
the submatrix of a matrix $M$ consisting of its $i_\ell$-th columns $\ell=1,2,\ldots,k$. 
For $j \in [n]$, define $n \times n$ matrix $B_j$ by
\[
B_j := (A_1u[j]\ A_2u[j]\ \cdots \ A_mu[j]). 
\]
Then we have
\[
 \sum_{k=1}^m (A_k u)e^{t \diag \lambda} (A_ku)^{\dagger} = 
 (B_1\  \cdots \ B_n)   
 \left(
 \begin{array}{ccc}
 e^{t\lambda_1 I_m} &  &     \\
     & \ddots    &      \\
     &      &    e^{t\lambda_nI_m}   
 \end{array}
 \right)
 \left(
 \begin{array}{c}
 B_1^{\dagger} \\
 \vdots \\
 B_n^{\dagger} 
 \end{array}
 \right).
\]
Let $B := (B_1\ B_2\ \cdots \ B_n)$. Define $mn \times mn$ 
diagonal matrix $D$ such that $D_{ii} := e^{t \lambda_{\lceil i/m \rceil}}$ for $i \in [mn]$.
Then $
\det \sum_{k=1}^m (A_k u)e^{t \diag \lambda} (A_ku)^{\dagger} = \det BDB^{\dagger}$.
Next, multiply to $B$ a product $L$  
of elementary matrices (other than that of scalar row multiplication)
 such that $\tilde B := LB$ is a row echelon form: 
For $i \in [n]$ there is $\rho(i) \in [nm]$ such that $\tilde B_{i\rho(i)} \neq 0$ and
$\tilde B_{i'j'} = 0$ for $i' \geq i$ and $j' < \rho(i)$.
By $\det \tilde BD\tilde B^{\dagger} = \det BDB^{\dagger}$ and Binet-Cauchy formula, it holds
\[
 \det BDB^{\dagger} = \sum_{i_1 < i_2 < \cdots < i_n} 
 |\det \tilde B[i_1,i_2,\ldots, i_n]|^2 e^{t (\lambda_{\lceil i_1/m \rceil} + \lambda_{\lceil i_2/m \rceil} + \cdots + \lambda_{\lceil i_n/m \rceil})}.
\]
Then, by $\lim_{t \to \infty} (1/t)\log \sum_{j} e^{a_j + b_jt} = \max_{j} b_j$, we have
\begin{eqnarray*}
\lim_{t \to \infty} \frac{1}{t} \log \det BDB^{\dagger} &=& \max \{ \lambda_{\lceil i_1/m \rceil}+\lambda_{\lceil i_2/m \rceil}+ \cdots + \lambda_
{\lceil i_n/m \rceil} \mid\det \tilde B[i_1,i_2,\ldots, i_n] \neq 0 \} \\
&=&  \lambda_{\lceil \rho(1)/m \rceil}+\lambda_{\lceil \rho(2)/m \rceil}+ \cdots + \lambda_
{\lceil\rho(n)/m \rceil} \\
&=& \sum_{i=1}^n \lambda_i |\{ j \in [n] \mid \rho(j) \in [(i-1)m+1, im] \}| \\
&=& \sum_{i=1}^n \lambda_i (\rank \tilde B[1,\ldots,im] - \rank \tilde B[1,\ldots,(i-1)m]) \\
&=& \sum_{i=1}^n (\lambda_i - \lambda_{i+1}) \rank B[1,\ldots,im] \\
&=& \sum_{i=1}^n (\lambda_i - \lambda_{i+1}) \rank (A_1u[1,\ldots,i]\  \cdots \ A_m u[1,\ldots,i]) .
\end{eqnarray*}
In the first equality, we use the facts that $\tilde B$ is the echelon form and that $i \mapsto \lambda_{\lceil i/m \rceil}$ is nondecreasing, where $\lceil r \rceil$ 
denotes the minimum integer at least $r$.
Since $\rank (A_1u[1,\ldots,i]\ \cdots \ A_mu[1,\ldots,i])  = \dim {\cal A}U_i$ and 
${\bf 1}^{\top} \lambda = \sum_i \lambda_i (i - (i-1)) = \sum_i (\lambda_i - \lambda_{i+1}) i = \sum_i (\lambda_i - \lambda_{i+1}) \dim U_i$,
we have the claim.
\end{proof}

\begin{Lem}\label{lem:inf_f^infty}
Suppose that ${\cal A}\CC^n = \CC^n$. Then it holds
\[
\inf_{H \in S_n: \| H\|_{\infty} \leq 1} f^{\infty}(H)  = 
2 \min_{U \leqslant \CC^n}  \dim {\cal A}U - \dim U.
\]
\end{Lem}
\begin{proof}
For $U \leqslant \CC^n$ with $k := \dim U$,
choose a unitary matrix $u$ containing a basis of $U$
in the first $k$ columns,
and define $\lambda_0 := (\overbrace{1, 1,\ldots,1}^{k},-1, -1, \ldots, -1)$.
Then, by Lemma~\ref{lem:f^infty}  (with $U_k = U$ and ${\cal A}U_n = U_n (=\CC^n)$), 
we have $f^{\infty}(u\diag \lambda_0 u^{\dagger}) 
= 2 (\dim {\cal A}U- \dim U)$. Thus $(\leq)$ holds.

For $H = u \diag \lambda u^{\dagger}$ with $\|\lambda\|_\infty \leq 1$, it holds
\begin{equation}
f^{\infty}(u \diag \lambda u^{\dagger}) = \sum_{i=1}^{n-1} (\lambda_i - \lambda_{i+1}) (  \dim {\cal A}U_i - \dim U_i),
\end{equation}
where we use $\dim {\cal A}U_n - \dim U_n =0$.
Choose an index $i^* \in [n]$ that minimizes $i \mapsto \dim {\cal A}U_i - \dim U_i \leq 0$.
Then we have $
f^{\infty}(u \diag \lambda u^{\dagger}) \geq \sum_{i=1}^{n-1} (\lambda_i - \lambda_{i+1}) (  \dim {\cal A}U_{i^*} - \dim U_{i^*}) \geq (\lambda_1- \lambda_n) (\dim {\cal A}U_{i^*} - \dim U_{i^*}) \geq 2 (\dim {\cal A}U_{i^*} - \dim U_{i^*})$. Hence $(\geq)$ holds.
\end{proof}

\begin{proof}[Proof of Theorem~\ref{thm:main}]
Let ${\cal A}^{\dagger} := (A_1^{\dagger}, A_2^{\dagger},\ldots, A_m^{\dagger})$.
We first consider the case 
where ${\cal A}\CC^n = \CC^n$ or  ${\cal A}^{\dagger}\CC^n = \CC^n$.
We may assume that ${\cal A}\CC^n = \CC^n$ (by replacing ${\cal A}$ by ${\cal A}^{\dagger}$ if necessary).
By Theorem~\ref{thm:duality} and Lemmas~\ref{lem:weakduality2}, \ref{lem:df}, \ref{lem:inf_f^infty}, we have
\begin{eqnarray}
&& \sup_{H\in S_n:\|H\|_{\infty} \leq 1} - f^{\infty}(H)  = \inf_{X \in P_n} \|df(X)\|^*_{\infty,X}= \inf_{g,h \in GL_n:T_{g^{\dagger}{\cal A}h}(I) = I} \|T^*_{g^{\dagger}{\cal A}h}(I) - I\|_{1} \nonumber \\
&&  \geq \inf_{g,h \in GL_n} \|T_{g^{\dagger}{\cal A}h}(I) - I\|_{1} + \|T_{g^{\dagger}{\cal A}h}^*(I) - I\|_{1} \nonumber \\
&& \geq 2 \max_{U \leqslant \CC^n} \dim U - \dim {\cal A}U = \sup_{H \in S_n:\|H\|_{\infty}\leq 1} - f^{\infty}(H).
\end{eqnarray}
Thus, all inequalities hold in equality. 
From the formula (Theorem~\ref{thm:FR}) of nc-rank, we have the claim.

Next consider the case where  $\dim {\cal A}\CC^n < n$ and  $\dim {\cal A}^{\dagger}\CC^n < n$.
Then there are $g,h \in GL_n$ such that
\[
g^{\dagger} A_k h = \left(  
\begin{array}{cc}
B_k & 0 \\
0   & 0
\end{array}\right) \quad (k \in [m])
\]
for $n' \times n'$ matrices $B_k$ $(k \in [m])$. Let ${\cal B} := (B_1,B_2,\ldots,B_m)$. 
We can assume that ${\cal B}\CC^{n'} = \CC^{n'}$ or ${\cal B}^{\dagger}\CC^{n'} = \CC^{n'}$.
In particular, (\ref{eqn:main}) holds for ${\cal B}$.
Since both sides of (\ref{eqn:main}) do not change under ${\cal A} \to g^{\dagger}{\cal A}h$,
we can assume from first that $A_k = \left(  
\begin{array}{cc}
B_k & 0 \\
0   & 0
\end{array}\right)$ for $k \in [m]$.
We claim:
\begin{equation}\label{eqn:AB}
\min_{U \leqslant \CC^n} \dim {\cal A}U - \dim U =  - (n-n') + \min_{V \leqslant \CC^{n'}} \dim {\cal B}V - \dim V. 
\end{equation}
Let $U_0$ be the coordinate vector subspace 
spanned by unit vectors $e_j$ for $j=n'+1,n'+2,\ldots,n$.
It is clear that  $\dim U_0 = n-n'$ and ${\cal A}U_0 = \{0\}$.
For $V \leqslant \CC^{n'}$, let $U := V + U_0$ (direct sum).
Then $\dim {\cal A}U - \dim U = - \dim {\cal B}V - \dim V - \dim U_0$, 
and we have ($\leq$).
Conversely, 
choose $U \leqslant \CC^n$ with minimum $\dim {\cal A}U - \dim U$.
It necessarily holds $U_0 \leqslant U$; otherwise 
replacing $U$ by $U + U_0$ decreases $\dim {\cal A}U - \dim U$.
Let $\pi:\RR^n \to \RR^{n'}$ denote the projection to 
the first $n'$ components. Let $V:= \pi U$. 
Then ${\cal A}U = {\cal B}V$. Also, since $\pi$ 
induces an injection $U/U_0 \to \pi(U) = V$, 
it holds $\dim V = \dim U - \dim U_0$.
From this, we have ($\geq$).

Now we have
\begin{eqnarray*}
&& 2 \max_{U \leqslant \CC^n} \dim U - {\cal A}U \leq \inf_{g,h \in GL_n} \|T_{g^{\dagger}{\cal A}h}(I) - I \|_1 + \|T^*_{g^{\dagger}{\cal A}h}(I) - I \|_1 \\
&& \leq 2(n-n') + \inf_{g',h' \in GL_{n'}} \|T_{(g')^{\dagger}{\cal B}h'}(I) - I \|_1 + \|T^*_{(g')^{\dagger}{\cal B}h'}(I) - I \|_1 \\
&& = 2(n-n') + 2 (\max_{V \leqslant \CC^{n'}} \dim V - \dim {\cal B}V) \\
&& = 2(n-n') + 2 \{- (n-n') + \max_{U \leqslant \CC^n} \dim U - \dim {\cal A}U\}. 
\end{eqnarray*}
Thus we obtain (\ref{eqn:main}) for ${\cal A}$.
Here, the first inequality follows from Lemma~\ref{lem:weakduality2}.
The second one is obtained by restricting $g,h \in GL_n$ 
to $g = \left(  
\begin{array}{cc}
g' & 0 \\
0   & I
\end{array}\right)$, $h = \left(  
\begin{array}{cc}
h' & 0 \\
0   & I
\end{array}\right)$ for $g',h' \in GL_{n'}$.
For the next equality, we use (\ref{eqn:main}) for ${\cal B}$ shown above. 
The final equality follows from (\ref{eqn:AB}). 
\end{proof}

\begin{Rem}
By these arguments, we see that the problem of 
minimizing the trace norm of the scaling residual 
\[
 {\rm (P)} \quad \mbox{Min.}  \quad \frac{1}{2} \left\{ \|T_{g^{\dagger}{\cal A}h} -I\|_1+  \|T^*_{g^{\dagger}{\cal A}h} -I\|_1  \right\} \quad 
    \mbox{s.t.}  \quad  g,h \in GL_n
\]
has the following dual problem: 
\begin{eqnarray*}
  {\rm (D)} \quad  \mbox{Max.} && - \sum_{i=1}^n (\lambda_i - \lambda_{i+1}) (  \dim {\cal A}U_i - \dim U_i) \\
    \mbox{s.t.} && \{0\} < U_1 < U_2 < \cdots < U_n = \CC^n, \\
    &&  1 \geq \lambda_1 \geq \lambda_2 \geq \cdots \geq \lambda_n \geq -1,
\end{eqnarray*}
where the notation $U < U'$ means that $U$ is a proper vector subspace of a vector space $U'$.
Then the optimal values of (P) and (D) are equal to the nc-rank.

As seen in Lemma~\ref{lem:df},
the objective function of the primal problem (P) is the norm of the differential $df$
of geodesically convex function $f$ on Hadamard manifold $P_n$.
For the scalable case, the case where $f$ is bounded below, 
the problem (P) is solved by minimizing $f$, 
as the operator Sinkhorn algorithm~\cite{GGOW,Gurvits2004} does so. 

On the other hand,
the dual problem (D) is also viewed 
as geodesically convex optimization on a {\em Hadamard space}---complete geodesic metric space having nonpositive curvature~\cite{BridsonHaefliger1999}.
Here the variable of (D) is a pair of
a complete flag of vector subspaces $U_i$ and nonincreasing weights $\lambda_i$.  
All such pairs $(\{U_i\}, \{\lambda_i\})$ form a so-called {\em Euclidean building}~\cite[Chapter II.10 Appendix]{BridsonHaefliger1999}, 
which is a representative ``non-manifold" Hadamard space.
Further, the objective function, which is the negative of the Lov\'asz extension of 
submodular function $U \mapsto \dim {\cal A}U - \dim U$,  
is a concave function on this space. 
These come from facts that 
the boundary of a Hadamard manifold becomes a Hadamard space and that
$f^{\infty}$ is a convex function on the boundary; see \cite{Hirai_Hadamard2022}.

The nc-rank computation algorithm by Hamada-Hirai~\cite{HamadaHirai2021}
solves a variant of the problem (D), and hence may be viewed 
as a dual algorithm, whereas the operator Sinkhorn may be viewed as a primal algorithm.
It is an interesting question whether the algorithm of Ivanyos, Qiao and Subrahmanyam~\cite{IQS2017,IQS2018}
has a primal-dual interpretation in this setting. 
\end{Rem}

\section*{Acknowledgments}
The author thanks Keiya Sakabe, Shin-ichi Ohta, Yuni Iwamasa, and Tasuku Soma for discussion, 
and thanks Cheng Zhuohao and the referees for helpful comments.
The author was supported by JSPS KAKENHI Grant Number JP24K21315.

\bibliographystyle{plain}
\bibliography{scaling_nc_rev}

\begin{thebibliography}{10}

\bibitem{AGS_GradientFlows}
L.~Ambrosio, N.~Gigli, and G.~Savar\'e.
\newblock {\em Gradient Flows in Metric Spaces and in the Space of Probability
  Measures}.
\newblock Birkh\"auser Verlag, Basel, second edition, 2008.

\bibitem{Amitsur}
S.~A. Amitsur.
\newblock Rational identities and applications to algebra and geometry.
\newblock {\em Journal of Algebra}, 3:304--359, 1966.

\bibitem{Bhatia2003}
R.~Bhatia.
\newblock On the exponential metric increasing property.
\newblock {\em Linear Algebra Appl.}, 375:211--220, 2003.

\bibitem{Bhatia_PositiveDefiniteMatrices}
R.~Bhatia.
\newblock {\em Positive Definite Matrices}.
\newblock Princeton University Press, Princeton, NJ, 2007.

\bibitem{Boumal_Book}
N.~Boumal.
\newblock {\em An Introduction to Optimization on Smooth Manifolds}.
\newblock Cambridge University Press, Cambridge, 2023.

\bibitem{BridsonHaefliger1999}
M.~R. Bridson and A.~Haefliger.
\newblock {\em Metric Spaces of Non-Positive Curvature}.
\newblock Springer-Verlag, Berlin, 1999.

\bibitem{BFGOWW_FOCS2019}
P.~B\"{u}rgisser, C.~Franks, A.~Garg, R.~Oliveira, M.~Walter, and A.~Wigderson.
\newblock Towards a theory of non-commutative optimization: geodesic 1st and
  2nd order methods for moment maps and polytopes.
\newblock In {\em 60th {IEEE} {A}nnual {S}ymposium on {F}oundations of
  {C}omputer {S}cience, FOCS 2019}, pages 845--861, 2019.

\bibitem{Cohn}
P.~M. Cohn.
\newblock {\em Skew Fields: Theory of General Division Rings}.
\newblock Cambridge University Press, Cambridge, 1995.

\bibitem{Davis1957}
C.~Davis.
\newblock All convex invariant functions of hermitian matrices.
\newblock {\em Arch. Math.}, 8:276--278, 1957.

\bibitem{Edmonds67}
J.~Edmonds.
\newblock Systems of distinct representatives and linear algebra.
\newblock {\em J. Res. Nat. Bur. Standards Sect. B}, 71B:241--245, 1967.

\bibitem{FortinReutenauer04}
M.~Fortin and C.~Reutenauer.
\newblock Commutative/non-commutative rank of linear matrices and subspaces of
  matrices of low rank.
\newblock {\em S\'em. Lothar. Combin.}, 52:B52f, 2004.

\bibitem{FriedlandFreitas2004}
S.~Friedland and P.~J. Freitas.
\newblock {$p$}-metrics on {${\rm GL}(n,\Bbb C)/{\rm U}_n$} and their
  {B}usemann compactifications.
\newblock {\em Linear Algebra Appl.}, 376:1--18, 2004.

\bibitem{GGOW}
A.~Garg, L.~Gurvits, R.~Oliveira, and A.~Wigderson.
\newblock Operator scaling: theory and applications.
\newblock {\em Found. Comput. Math.}, 20:223--290, 2020.

\bibitem{Gurvits2004}
L.~Gurvits.
\newblock Classical complexity and quantum entanglement.
\newblock {\em J. Comput. System Sci.}, 69:448--484, 2004.

\bibitem{HamadaHirai2021}
M.~Hamada and H.~Hirai.
\newblock Computing the nc-rank via discrete convex optimization on {${\rm
  CAT}(0)$} spaces.
\newblock {\em SIAM J. Appl. Algebra Geom.}, 5:455--478, 2021.

\bibitem{HayashiHiraiSakabe}
K.~Hayashi, H.~Hirai, and K.~Sakabe.
\newblock Finding {H}all blockers by matrix scaling.
\newblock {\em Math. Oper. Res.}, 49:2166--2179, 2024.

\bibitem{Hirai_Hadamard2022}
H.~Hirai.
\newblock Convex analysis on {H}adamard spaces and scaling problems.
\newblock {\em Found. Comput. Math.}, 24:1979--2016, 2024.

\bibitem{HiraiSakabe2024FOCS}
H.~Hirai and K.~Sakabe.
\newblock Gradient descent for unbounded convex functions on {H}adamard
  manifolds and its applications to scaling problems.
\newblock In {\em 65th {IEEE} Annual Symposium on Foundations of Computer
  Science, {FOCS} 2024}, pages 2387--2402. {IEEE}, 2024.
\newblock {\tt arXiv:2404.09746}.

\bibitem{IQS2017}
G.~Ivanyos, Y.~Qiao, and K.~V. Subrahmanyam.
\newblock Non-commutative {E}dmonds' problem and matrix semi-invariants.
\newblock {\em Comput. Complex.}, 26:717--763, 2017.

\bibitem{IQS2018}
G.~Ivanyos, Y.~Qiao, and K.~V. Subrahmanyam.
\newblock Constructive non-commutative rank computation is in deterministic
  polynomial time.
\newblock {\em Comput. Complex.}, 27:561--593, 2018.

\bibitem{Kabanets2004}
V.~Kabanets and R.~Impagliazzo.
\newblock Derandomizing polynomial identity tests means proving circuit lower
  bounds.
\newblock {\em Comput. Complex.}, 13:1--46, 2004.

\bibitem{KLM2009JDG}
M.~Kapovich, B.~Leeb, and J.~Millson.
\newblock Convex functions on symmetric spaces, side lengths of polygons and
  the stability inequalities for weighted configurations at infinity.
\newblock {\em J. Differential Geom.}, 81:297--354, 2009.

\bibitem{KleinerLeeb2006}
B.~Kleiner and B.~Leeb.
\newblock Rigidity of invariant convex sets in symmetric spaces.
\newblock {\em Invent. Math.}, 163(3):657--676, 2006.

\bibitem{Lewis1996}
A.~S. Lewis.
\newblock Convex analysis on the {H}ermitian matrices.
\newblock {\em SIAM J. Optim.}, 6:164--177, 1996.

\bibitem{Lovasz89}
L.~Lov\'asz.
\newblock Singular spaces of matrices and their application in combinatorics.
\newblock {\em Bol. Soc. Brasil. Mat. (N.S.)}, 20:87--99, 1989.

\bibitem{Rockafellar}
R.~T. Rockafellar.
\newblock {\em Convex Analysis}.
\newblock Princeton University Press, Princeton, NJ, 1970.

\bibitem{Sakai1996}
T.~Sakai.
\newblock {\em Riemannian Geometry}.
\newblock American Mathematical Society, Providence, RI, 1996.

\bibitem{Sinkhorn1964}
R.~Sinkhorn.
\newblock A relationship between arbitrary positive matrices and doubly
  stochastic matrices.
\newblock {\em Ann. Math. Statist.}, 35:876--879, 1964.

\end{thebibliography}

\appendix
\section{Proof of Lemma~\ref{lem:norm}}
(2). 
For $H = u \diag \lambda u^{\dagger}$, it holds
\begin{eqnarray*}
\| H \|_{v}^* &=& \max \{ \trace (\diag \lambda) u^{\dagger} X u \mid X \in S_n:\|X\|_{v} \leq 1 \} \\
&=& \max \{ \trace (\diag \lambda) Y \mid Y \in S_n:\|Y\|_{v} \leq 1 \} \\
&=& \max \{ \lambda^{\top}  D(Y) \mid Y \in S_n:\|Y\|_{v} \leq 1 \}. 
\end{eqnarray*}
Here, by restricting the range to diagonal matrices $Y = \diag \mu$, we have
\[
\| H \|_{v}^* \geq \max \{ \lambda^{\top} \mu \mid v(\mu) \leq 1\} = v^*(\lambda) = \|H\|_{v^*}.
\]
On the other hand, by the Schur-Horn theorem, 
$D(Y)$ belongs to the permutation polytope of the vector $\mu$ of eigenvalues of $Y$.
Thus, we have
\[
 \lambda^{\top} D(Y) \leq \max \{ \lambda^{\top} \sigma \mu \mid \sigma:\mbox{permutation matrix} \}.
\]
Thus 
\[
\| H \|_{v}^* \leq \max \{\lambda^{\top}\mu \mid v(\mu) \leq 1\} = v^*(\lambda) = \|H\|_{v^*}.
\]

(3).
Let $H = u \diag \lambda u^*$ for $u \in U_n$ and $\lambda \in \RR^n$.
Then 
\[
\|H\|_v = v(\sigma \lambda) \quad (\forall \sigma:\mbox{permutation matrix}).
\]
By the Schur-Horn theorem and the convexity (of sublevel sets) of $v$, it holds
\[
D(H) \in \mbox{convex hull of }\{\sigma \lambda \mid \sigma:\mbox{permutation matrix}\} \subseteq \{x \in \RR^n \mid v(x) \leq \|H\|_v \}. 
\]
This implies  $v(D(H)) \leq \|H\|_v$.

(1).
    It suffices to show $\|X+Y\|_v \leq \|X\|_v+\|Y\|_v$.
    Let $Z := X+Y$ and suppose that 
    $Z = u \diag \lambda u^{\dagger}$ for $u \in U_n$ and $\lambda \in \RR^n$.
    Let $\tilde X := u^{\dagger} X u$ and $\tilde Y := u^{\dagger} Y u$.
    By
    $\diag \lambda = D(\tilde X) + D(\tilde Y)$ and (2), 
    we have
     \[
     \|X+Y\|_v = v(D(\tilde X) + D(\tilde Y)) \leq v( D(\tilde X)) + v( D(\tilde Y)) \leq \|\tilde X\|_v +\|\tilde Y\|_v =\|X\|_v +\|Y\|_v.  
     \]

\end{document}